\documentclass[12pt]{article}
\usepackage{amsmath,amssymb, amsfonts, amsthm,amscd}
\usepackage[T2A]{fontenc}
\usepackage[cp1251]{inputenc}
\usepackage[english]{babel}
\usepackage{graphicx}

\theoremstyle{plain}
\newtheorem{theorem}{Theorem}
\newtheorem{lemma}{Lemma}
\newtheorem{proposition}{Proposition}
\newtheorem{corollary}{Corollary}
\theoremstyle{definition}
\newtheorem{definition}{Definition}

\newtheorem{remark}{Remark}

\begin{document}

\begin{center}
{\Huge Rings of simple range 2}
\end{center}
\vskip 0.1cm \centerline{{\large \fbox{ B.V. Zabavsky}, O.M. Romaniv, A.V. Sagan}}

\vskip 1cm

\footnotesize{\noindent\textbf{Corresponding author:} \textit{O.M. Romaniv} }

\vskip 0.3cm

\footnotesize{\noindent\textbf{Address:} \textit{Ivan Franko National University of Lviv,  Ukraine (O.M. Romaniv, A.V. Sagan)} }

\vskip 0.3cm

\footnotesize{\noindent\textbf{Abstract:} \textit{We introduce the concept of rings of simple range 2. Based on this concept, we build a theory diagonal reduction of matrices over Bezout domain. In particular we show that invariant Bezout domain is an elementary divisor ring if and only if it is a rings of simple range 2.} }

\vskip 0.3cm

\footnotesize{\noindent\textbf{Keywords and phrases:} \textit{Bezout ring, Hermite ring, stable range, diagonal reduction, simple ring, simple range}}

\vskip 0.3cm

\noindent{\textbf{Mathematics Subject Classification}}: 06F20, 13F99, 13E15

\vspace{1truecm}

\normalsize

\section{Introduction}

The study of linear groups over fields goes back to the mid-19th century. However, practically nothing was known about the ring except for semilocal and some types of arithmetic rings until the mid-1960s. A real commenting revolution was initiated by work H.~Bass~\cite{1-bass}. In particular, H.~Bass introduced a new concept of dimension of rings, a stable range and found that the main results on the structure of complete linear group $GL_n$ over the field with the necessary changes are transferred to rings, whose stable range is less than~$n$.  Note that a stable range is also closely related to the issues of stable structure of projective modules. In modern algebraic research there is a close connection between the concept of a stable range of ring and the open questions of diagonalization of matrices over these rings. The problem of diagonalization of matrices is a classic one. An overview can be found in~\cite{2-zabavsky}. It is known~\cite[Theorem~1.2.40]{2-zabavsky} that the stable range of an elementary divisor ring did not exceed~2. I.~Kaplansky noted that a commutative Bezout ring of stable range~1 is an elementary divisor ring~\cite{3-kaplansky}. M.~Henriksen showed that a unit regular ring is a ring over which the matrices are diagonalized~\cite{4-henri}. Even in the case of commutative rings, the bound of stable range is found to be cramped. In particular~\cite{12-zabrom,5-zabavsky,6-zabavsky} introduced a generalization of stable range of rings (rings of Gelfand range~1, rings of dyadic range~1) which allowed to reduce tha problem of complete description of the elementary divisor rings to the question of existence of a nontrivial Gelfand element~\cite{5-zabavsky}. In this paper we introduce the concept of a ring of simple range~2, which allows us to describe the new classes of noncommutative elementary divisor rings.

\section{Notations and preliminary results}

Let $R$ be an   associative ring with  non-zero unit.
We say that matrices $A$ and $B$ over a ring $R$ are equivalent if there  exist invertible matrices $P$ and $Q$ over  $R$ such that $B=PAQ$. If   a matrix $A$ over $R$ is equivalent to a diagonal matrix $D=(d_{ii})$ with the property that $d_{ii}$ is a total divisor of $d_{i+1,i+1}$ (i.e. $Rd_{i+1,i+1}R\subset d_{i,i}R\cap Rd_{i,i}$), then  we say that $A$  admits a canonical diagonal reduction. A ring $R$ over which every  matrix admits a  canonical diagonal reduction is called an {\it elementary divisor ring}. A ring $R$ is called {\it  right (left) Hermite} if each matrix $A\in R^{1\times 2}$ ($A\in R^{2\times 1}$) admits a diagonal reduction.  A ring which is  right and  left  Hermite is called a {\it Hermite ring}. Obviously a commutative left (right) Hermite ring is a Hermite ring. Moreover, each  elementary divisor ring is Hermite, and a right (left) Hermite ring is a right (left) \textit{B\'ezout ring}, i.e. a ring in which any   finitely generated right (left) ideal of $R$  is principal \cite{2-zabavsky}.

A row $(a_1,\ldots, a_n)\in R^{n}$  is called {\it unimodular} if $a_1R+a_2R+\cdots+a_nR=R$.
An unimodular $n$-row $(a_1,\ldots, a_n)\in R^{n}$   over a ring $R$  is called {\it reducible} if there exist a $(n-1)$-row  $(b_1,\ldots, b_{n-1})\in R^{n-1}$ such that
\[
(a_1+a_nb_1,a_2+a_nb_2,\ldots, a_{n-1}+a_nb_{n-1})\in R^{n-1}
\]
is unimodular. If $n\in\mathbb{N}$ is the smallest number  such that any  unimodular $(n+1)$-row is reducible, then  $R$ has {\it stable range} $n$, where $n\geq 2$. In particular, a ring $R$ has  {\it stable range} 1 if $aR + bR = R$ implies that $(a + bt)R = R$ for some $t\in R$. A ring $R$ is {\it a ring of stable range} 2 if $aR + bR+cR = R$ implies that $(a + cx)R+(b+cy)R= R$ for some $x,y\in R$. A right (left) Hermite ring is a ring of stable range~2 \cite[Theorem 1.2.40]{2-zabavsky}. In the case of commutative Bezout we have a results.

\begin{theorem}\cite[Theorem 2.1.2]{2-zabavsky}\label{theor-2.1}
 A commutative Bezout ring is a Hermite ring if and only if it is a ring of stable range~2.
\end{theorem}

This is an open question: is right (left) Bezout ring of stable range 2 a right (left) Hermite ring?

By \cite{7-amitsur} we have that Bezout domain is a Hermite ring, i.e. Bezout domain is a ring of stable range~2.

In the future we will consider  rings in which for any $a\in R$ we have $RaR=a_*R=Ra_*$ for some $a_*\in R$. Examples of such ring are simple rings \cite[Section 4.2]{2-zabavsky}, quasi-duo elementary divisor rings, semi-local semi-prime elementary divisor rings \cite[Theore~1]{9-dubrovin}, principal ideal domains. Condition $RaR=a_*R=Ra_*$ for any $a\in R$ is called \textit{Dubrovin condition} \cite{2-zabavsky}. Recall that nonzero element $a\in R$ is right (left) invariant if $aR$ ($Ra$) is two-sided ideal, if both conditions hold that is, if $aR=Ra$ then $a$ is called invariant element. If in a domain $R$ every factor of an invariant element is invariant we say that $R$ is a domain with Komarnytsky condition.  Obvious example of domain with Komarnytsky condition is a simple domain and domain in which any element is an invariant. Note that a principal ideal domain $\mathbb{H}[x]$, where $\mathbb{H}$ is quaternion divisor ring, we have invariant element $x^2+1$ but not any factor of it is not invariant. D-K elementary divisor rings are closely related with rings with Dubrovin and Komarnytsky conditions (hence forth, rings  with Dubrovin-Komarnytsky (D-K) condition).

\begin{definition}\label{def-2.1}
Ring $R$ is called an \textit{D-K elementary divisor ring} if an arbitrary matrix $A$ over $R$ is the equivalent matrix $\mathrm{diag} (\varepsilon_1,\varepsilon_2,\dots,\varepsilon_r,0,\dots,0)$ where $R\varepsilon_{i+1}R\subset \varepsilon_iR\cap R\varepsilon_i$ for all $i\in 1,r-1$ and $\varepsilon_1$, $\varepsilon_2$, \dots, $\varepsilon_{r-1}$ are invariant elements.
\end{definition}

A simple elementary divisor domain is an example of D-K elementary divisor ring.  By \cite[Theorem 4.1.1]{2-zabavsky}  a simple Bezout domain is an elementary divisor ring if and only if it is 2-simple domain. i.e. for any $a\in R\backslash\{0\}$ we have that $u_1av_1+u_2av_2=1$ for some $u_1,u_2,v_1,v_2\in R$. Consequently we have a next result.

\begin{corollary}
  Simple Bezout domain is D-K elementary divisor ring if and only if it is 2-simple domain.
\end{corollary}

Let $R$ be a simple ring. Clearly $RaR=\{\sum_{i=1}^n u_iav_i\mid u_i,v_i\in R\}=R$ for each $a\in R\backslash\{0\}$. Then we have $n\in \mathbb{N}$ such that
\begin{equation}\label{*}
  u_1av_1+u_2av_2+\dots+u_nav_n=1.
\end{equation}
If for all $a\in R\backslash\{0\}$ there exists a minimal $n$ which satisfies~\eqref{*}, then $R$ is called \textit{$n$-simple}.

Description D-K elementary divisor ring is quite an interesting and extremes challenging task for future usearch. Just note the next open task: is there $n$-simple Bezout domain where $n\geq 3$?

\section{Rings of simple range 2}

\begin{definition}
We say a ring $R$ is \textit{a ring of simple range 2} if $RaR+RbR+RcR=R$ where $c\ne0$ implied $R(pa+qb)R+RpcR=R$ for some $p,q\in R$.
\end{definition}

\begin{proposition}\label{prop-3.1}
  Let $R$ be Bezout ring. The following statements are equivalent:

  \begin{description}
    \item[1)] $R$ is a ring of simple range 2;
    \item[2)] $RaR+RbR+RcR$, $c\ne0$, implied $(pa+qb)R+pcR=dR$ where $RdR=R$ for some $p,q\in R$.
  \end{description}
\end{proposition}
\begin{proof}
Let $RaR+RbR+RcR=R$, $c\ne0$, implied $(pa+qb)R+pcR=dR$ where $RdR=R$ for some $p,q\in R$. Since $pcR\subseteq RpcR$ and $(pa+qb)R\subseteq R(pa+qb)R$ we have $(pa+qb)R+pcR\subseteq R(pa+qb)R+RpcR$. So on $RdR=R$ and $d\in R(pa+qb)R+RpcR$ we have $R(pa+qb)R+RpcR=R$.

Let condition $RaR+RbR+RcR=R$, $c\ne0$, implied $R(pa+qb)R+RpcR=R$ for some $p,q\in R$. Since $R$ is a Bezout ring then $(pa+qb)R+pcR=dR$ for some $d\in R$. Since $(pa+qb)R\subset dR$ and $pcR\subseteq RdR$ and $R(pa+qb)R\subseteq RdR$. So on $R(pa+qb)R+RpcR=R$ we have that $RdR=R$.
\end{proof}

Obviously a simple ring is  ring of simple range 2.

\medskip

A ring $R$ is called \textit{purely infinite} if $R\cong R^2$ as $R$-modules.
For example for any ring $S=End_R(R^\infty)$ is purely infinite, where $R^\infty$ denotes an infinite direct product of $R$. This ring is simple ring, i.e. a ring of simple range 2 which have infinite stable range. Nevertheless we have a next result.

\begin{theorem}
  A Bezout ring of stable range 1 is a ring of simple range 2.
\end{theorem}
\begin{proof}
  Let $RaR+RbR+RcR=R$, $c\ne0$. Since $R$ is a Bezout ring we have $aR+bR=dR$, $a=da_1$, $b=db_1$ and $au+bv=d$ for some $d,a_1, b_1, u,v\in R$. Then $d(1-a_1u-b_1v)=0$. Let $c=1-a_1u-b_1v$ we have $a_1R+b_1R+cR=R$ and $dc=0$. Since any ring of stable range 1 is a ring of stable range 2 we have that $(a_1+c\lambda)R+(b_1+c\mu)R=R$ for some $\lambda, \mu\in R$. Let $a_0=a_1+c\lambda$, $b_0=b_1+c\mu$, then $a_oR+b_0R=R$ and $a=da_0$, $b=db_0$. Since $R$ is a ring of stable range 1 we have $(b_0\lambda +a_0)R=R$. i.e. $b_0\lambda +a_0=u$ is invertible element $R$. Since $RaR+RbR+RcR=R$ and $b\lambda +a=du$ we have $(b\lambda+a)R=duR=dR$ and $(b\lambda+a)R\subseteq R(b\lambda+a)R$, $dR\subseteq RdR$ and $RaR+RbR=RdR$. Then $R(b\lambda+a)R+RcR=RaR+RbR+RcR=R$. i.e. $R$ is a ring of simple range 2.
\end{proof}

\begin{theorem}\label{theor-3.3}
  D-K elementary divisor domain is a ring of simple range 2.
\end{theorem}
\begin{proof}
  Let $R$ be a D-K elementary divisor domain and $RaR+RbR+RcR=R$, $c\ne 0$. Consider a matrix $A=\Bigl(\begin{smallmatrix}
                             a & c \\[3pt]
                             b & 0
                           \end{smallmatrix}\Bigr)$.
  Let $b\ne 0$, then
  \begin{equation}\label{**}
    PAQ=\begin{pmatrix}
                             z & 0\\
                             0 & \alpha
                           \end{pmatrix}
  \end{equation}
  where $R\alpha R\subset zR\cap Rz$, for some invertible matrices $P$ and $Q$. Since $b\ne0$ we have that $\alpha\ne 0$. Since $RaR+RbR+RcR=R$ and $R\alpha R\subset aR\cap Rz$ according to \eqref{**} we have $R=RaR+RbR+RcR=zR=Rz$, i.e. $z$ is an invertible element $R$. We denote
  $P=\Bigl(\begin{smallmatrix}p & q \\[3pt] * & *\end{smallmatrix}\Bigr)$. According to \eqref{**} we have $(pa+qb)R+pcR=R$. By Proposition~\ref{prop-3.1} we have $R(pa+qb)R+RpcR=R$. If $b=0$ we have $R(a+b)R+RcR=R$. We proved that $R$ is a ring of simple range 2.
\end{proof}

We will note that if $R$ is a Bezout domain and $RaR=R$ and if $aR\subset bR$ or $Ra\subset Rc$ we have that $RbR=R$ and $RcR=R$. We have a next result.

\begin{proposition}\label{prop-3.4}
  Let $R$ be a Bezout domain with D-K condition. If $RaR=R$ and $RbR=R$ then $RabR=R$.
\end{proposition}
\begin{proof}
  In $R$ there exists $a,b\in R$ such that $RaR=R$ and $RbR=R$ such that $RabR=\alpha R=R\alpha$ where $\alpha$ is not invertible element of $R$. Since $R$ is a Bezout domain we have $bR+\alpha R=\beta R$ for some $\beta\in R$. Since $\alpha R\subset \beta R$ we have that $\beta$ is an invariant element. Since $bR\subset\beta R$ and $RbR=R$ we have that $\beta$ is invertible element $R$.

  Then we have that $bu+\alpha v=1$ for some $u,v\in R$. Since $abu+a\alpha=a$ and $RabR=\alpha R=R\alpha$ we have that $ab=s\alpha$ for some $s\in R$. Then $su'\alpha+a\alpha=a$ where $\alpha u=u'\alpha$. So on we have $Ra\subset R\alpha$ and $R=RaR\subset R\alpha R$ we have that $\alpha$ is an invertible element. The contradiction obtained proves Proposition~\ref{prop-3.4}.
\end{proof}

Let $R$ be a Bezout domain of simple range 2, i.e. $RaR+RbR+RcR=R$, $c\ne 0$, imply $R(pa+qb)R+RpcR=R$.

\begin{remark}\label{remark-1}
  According to Proposition~\ref{prop-3.1}, Proposition~\ref{prop-3.4} and the fact then $R$ is Bezout domain we can assume that it is equality $R(pa+qb)R+RpcR=R$ elements $p$, $q$ such that $pR+qR=R$. If $pR+qR=tR$, then $p=tp_0$, $q=tq_0$, $p_0R+q_0R=R$ for some $t, p_0, q_0\in R$. Since $R(pa+qb)R+RpcR=R$ we have $RtR=R$ and $R(p_0a+q_0b)R+Rp_0cR=R$.
 \end{remark}

We  will need the next ancient result in the future.

\begin{proposition}\cite[Lemma 1]{10-williams}\label{prop-3.5}
  Let $R$ be a Bezout domain and $pR+qR=R$, $Ru+Rv=R$ for some $p,q,u,v\in R$. Then there exist invertible matrices $P$ and $Q$ of the form
  $$
  P=\begin{pmatrix}
      p & q \\
      * & *
    \end{pmatrix}, \quad Q=\begin{pmatrix}
      u & * \\
      v & *
    \end{pmatrix}.
  $$
\end{proposition}

As a result we get the main result of this paper in terms of constructing the theorem of matrix diagonalization over noncommutative rings.

\begin{lemma}\label{lemma-3.6}
  Let $R$ be a Bezout domain of simple range 2. Then any matrix $A=\Bigl(\begin{smallmatrix}
      a & c \\[3pt] 
      b & 0
    \end{smallmatrix}\Bigr)$ where $c\ne0$ and $RaR+RbR+RcR=R$ is equivalent to matrix $X=\Bigl(\begin{smallmatrix}
      x & z \\[3pt] 
      y & 0
    \end{smallmatrix}\Bigr)$ such that $RxR=R$.
\end{lemma}

\begin{proof}
  Since $RaR+RbR+RcR=R$ and $R$ is a ring of simple range 2, then $(pa+qb)R+pcR=R$ where $RxR=R$. We have
  $$
  \begin{pmatrix}
      p & q
    \end{pmatrix}\begin{pmatrix}
      a & c \\
      b & 0
    \end{pmatrix}\begin{pmatrix}
      u \\
      v
    \end{pmatrix}=x
  $$
  where $(pa+qb)u+pcv=x$ for some $u,v\in R$. By Remark~\ref{remark-1} and Proposition~\ref{prop-3.5} we have
  $$
  PAQ=\begin{pmatrix}
      x & * \\
      * & *
    \end{pmatrix}
  $$ for some invertible matrices
  $$
  P=\begin{pmatrix}
      p & q \\
      * & *
    \end{pmatrix}, \quad Q=\begin{pmatrix}
      u & * \\
      v & *
    \end{pmatrix}.
  $$
\end{proof}

According to Theorem~\ref{theor-3.3} D-K elementary divisor domain is a ring of simple range 2. In what follows we shall prove that for invariant Bezout domain the condition of a simple range 2 is sufficient, that is, the following theorem holds.

\begin{theorem}\label{theor-3.7}
  Let $R$ be invariant Bezout domain of simple range 2. Then $R$ is D-K elementary divisor ring.
\end{theorem}
\begin{proof}
 Invariant Bezout domain is obvious example of a domain with D-K conditions. According to \cite[Proposition 1]{11-gatzab}  in order to prove Theorem~\ref{theor-3.7} we  need to prove that matrix
 $A=\Bigl(\begin{smallmatrix}
      a & c \\[3pt]
      b & 0
    \end{smallmatrix}\Bigr)$
    where $c\ne0$ and $RaR+RbR+RcR=R$ hold canonical diagonal reduction. By Lemma~\ref{lemma-3.6}, matrix $A$ is equivalent to matrix
$X=\Bigl(\begin{smallmatrix}
      x & y \\[3pt]
      0 & z
    \end{smallmatrix}\Bigr)$   where $RxR=R$. Since $R$ is an invariant domain we have that $x$ is an invertible element. Elementary transformations of row and columns matrix $X$ will be reduced to appearance
$D=\Bigl(\begin{smallmatrix}
      1 & 0 \\[3pt] 
      0 & \Delta
    \end{smallmatrix}\Bigr)$    for some $d\in R$, i.e. we proved that a matrix $A$ hold a canonical diagonal reduction.
 \end{proof}

\end{document}